\documentclass[12pt,leqno]{amsart}
\usepackage{pifont}
\usepackage{stmaryrd}
\usepackage{dsfont}
\usepackage{color}
\usepackage{mathrsfs}
\usepackage{amsmath}
\usepackage{amssymb}
\usepackage[hidelinks]{hyperref}
\usepackage{bookmark}
\usepackage[bottom]{footmisc}
\usepackage{verbatim}
\usepackage{extarrows}
\usepackage{mathtools}
\usepackage{orcidlink}

\numberwithin{equation}{section}
\newtheorem{thm}{Theorem}[section]
\newtheorem{lem}[thm]{Lemma}

\newtheorem{defi}[thm]{Definition}
\newtheorem{rmk}[thm]{Remark}

\newcommand{\wt}{\operatorname{wt}}
\newcommand{\xqz}[1]{\lfloor #1 \rfloor}

\renewcommand{\mod}{\operatorname{mod}}

\def\Z{\mathbb{Z}}
\def\N{\mathbb{N}}

\allowdisplaybreaks

\begin{document}

\title{Twisted bimodules and associative algebras associated to VOAs}

\author{Shun Xu~\orcidlink{0009-0006-8080-8107}}

\address{School of Mathematical Sciences, Tongji University, Shanghai, 200092, China}

\email{shunxu@tongji.edu.cn}



\begin{abstract} 
Let $V$ be a vertex operator algebra, $g$ be an automorphism of $V$ of order $T$, and  $m, n \in (1/T)\mathbb{N}$.  
In~\cite{HX2} and~\cite{HXX1}, it was shown respectively that the associative algebra $A_{g,n}(V)$ constructed by Dong, Li, and Mason~\cite{DLM3},  
and the $A_{g,n}(V)$--$A_{g,m}(V)$-bimodule $A_{g,n,m}(V)$ constructed by Dong and Jiang~\cite{DJ2}, are both isomorphic to certain subquotients of $U(V[g])$,  
where $U(V[g])$ denotes the universal enveloping algebra of $V$ with respect to $g$.  
In this paper, we give a unified and concise proof of these isomorphisms.
\end{abstract}

\maketitle

\section{Introduction}

Let $V$ be a vertex operator algebra and let $g$ be an automorphism of $V$ of finite order $T$.  
To study the twisted representation theory of $V$, various associative algebras have been introduced. In this paper, we focus on two principal constructions: one is Zhu's algebra $A(V)$ and its generalizations, and the other is the universal enveloping algebra $U(V[g])$ associated with the $g$-twisted structure.

For $n, m \in (1/T)\mathbb{N}$, Dong, Li, and Mason constructed a family of associative algebras $A_{g,n}(V)$ in \cite{DLM3}, where $A_{g,0}(V) = A(V)$ recovers Zhu's original algebra \cite{Z1}. Subsequently, Dong and Jiang introduced a family of $A_{g,n}(V)$--$A_{g,m}(V)$-bimodules $A_{g,n,m}(V)$ in \cite{DJ2}.

On the other hand, the universal enveloping algebra $U(V[g])$ of $V$ with respect to $g$ is a $(1/T)\mathbb{Z}$-graded associative algebra. For $n, m \in (1/T)\mathbb{N}$, one can construct from $U(V[g])$ the quotient algebras 
$
U(V[g])_0 / U(V[g])_0^{-n - 1/T}
$
and the quotient spaces
$
U(V[g])_{n-m} / U(V[g])_{n-m}^{-m - 1/T},
$
which carries the structure of a 
$
 U(V[g])_0 / U(V[g])_0^{-n - 1/T} \text{--} U(V[g])_0 / U(V[g])_0^{-m - 1/T} 
$-bimodule. The precise relationship between these constructions is captured by the following theorem.

\begin{thm}\label{thm1.1}
For any $n,m \in (1/T)\mathbb{N}$, define a linear map
\[
\varphi_{n, m}: A_{g,n, m}(V) \longrightarrow U(V[g])_{n-m} / U(V[g])_{n-m}^{-m - 1/T}
\]
by sending $u + O_{n, m}(V)$ to $J_{m-n}(u) + U(V[g])_{n-m}^{-m - 1/T}$.  
Then $\varphi_{n, n}$ is an algebra isomorphism, and more generally, $\varphi_{n, m}$ is an $A_{g,n}(V)$--$A_{g,m}(V)$-bimodule isomorphism.
\end{thm}

It was shown by Han and Xiao \cite{HX2} that the associative algebra $A_{g,n}(V)$ is isomorphic to $U(V[g])_0 / U(V[g])_0^{-n - 1/T}$, and by the author together with Han and Xiao \cite{HXX1} that the bimodule $A_{g,n,m}(V)$ is isomorphic to $U(V[g])_{n-m} / U(V[g])_{n-m}^{-m - 1/T}$. In this paper, we present a unified and streamlined proof of these isomorphisms using the method developed in \cite{Shun3, Shun2}.

The untwisted case of these results for vertex operator algebras can be found in \cite{DLM1,DJ1,H1,H3}; their extensions to the setting of vertex operator superalgebras are discussed in \cite{KW1,JJ1,Shun2,Shun3}.

The paper is organized as follows.  
In Section~2, we recall the definitions of vertex operator algebras, weak $g$-twisted modules, and admissible $g$-twisted modules.  
In Section~3, we review the $A_{g,n}(V)$--$A_{g,m}(V)$-bimodules $A_{g,n,m}(V)$ and explain their representation-theoretic significance. In Section~4,  we recall the construction of the universal enveloping algebra  $U(V[g])$ of $V$ with respect to $g$.
In Section~5, we give the proof of Theorem~\ref{thm1.1}.

\section{Basics}

We recall definitions of the vertex operator algebras, weak $g$-twisted modules and admissible $g$-twisted modules in this
section.

\begin{defi}
A vertex operator algebra is a quadruple $(V, Y, \mathbf{1}, \omega)$, where $V=\bigoplus_{n \in \mathbb{Z}} V_n$ is a $\mathbb{Z}$-graded vector space with $\operatorname{dim} V_n<\infty$ for all $n \in \mathbb{Z}, V_n=0$ for $n \ll 0, \mathbf{1} \in V_0$, $\omega \in V_2$ and $Y$ is a linear map from $V$ to $ (\operatorname{End}V)\left[\left[z, z^{-1}\right]\right]$ sending $v \in V$ to $Y(v, z)=$ $\sum_{n \in \mathbb{Z}} v_n z^{-n-1}$, such that the following conditions are assumed for any $u, v \in V$ :

{\rm (1) }$Y(\mathbf{1}, z)=\operatorname{id}_V$ and $u_n \mathbf{1}=\delta_{n,-1} u$ for $n \geq-1$.

{\rm(2)} $u_n v=0$ for $n \gg 0$.

{\rm(3)} For any $l, m, n \in \mathbb{Z}$,  the Jacobi identity hold: 
$$
\sum_{i \geq 0}(-1)^i\binom{l}{i}\left(u_{m+l-i} v_{n+i}-(-1)^l v_{n+l-i} u_{m+i}\right)=\sum_{i \geq 0}\binom{m}{i}\left(u_{l+i} v\right)_{m+n-i}.
$$

{\rm(4)} The Virasoro algebra relations hold: $[L(m), L(n)]=(m-n) L(m+n)+\frac{m^3-m}{12} c_V$ for $m, n \in \mathbb{Z}$, where $c_V \in \mathbb{C}$ and $L(m)=\omega_{m+1}$ for $m \in \mathbb{Z} ;\left.L(0)\right|_{V_m}=m \mathrm{id}_{V_m}$ for $m \in \mathbb{Z}$ and $Y(L(-1) w, z)=\frac{d}{d z} Y(w, z)$ for $w \in V$.
\end{defi}

\begin{defi}
    Let $(V, Y, \mathbf{1}, \omega)$ be a vertex operator algebra. A linear isomorphism $g$ of $V$ is called an automorphism of $V$ if
$$
g(\mathbf{1})=\mathbf{1}, g(\omega)=\omega \quad \text { and } \quad g(Y(u, z) v)=Y(g(u), z) g(v) \quad \text { for } u, v \in V.
$$
\end{defi}
Let $(V,Y,\mathbf{1},\omega)$ be a vertex operator algebra.  For any $n$, elements in $V_n$ are called homogeneous  and if $u\in V_n$  we define $\operatorname{wt}u=n$.  So when  $\operatorname{wt}u$ appears we always assume that $u$ is homogeneous. We fix $g$ to be an automorphism of $V$ of finite order $T$. Then $V$ has the following decomposition with respect to $g$ :
$$
V=\bigoplus_{r=0}^{T-1} V^r, \quad \text { where } V^r=\left\{v \in V \mid g v=e^{-2 \pi \sqrt{-1} r / T} v\right\}.
$$

\begin{defi}
 A weak $g$-twisted $V$-module  is a vector space $M$  equipped with a linear map $Y_M(\cdot, z)$ from $V$ to $(\operatorname{End} M)\left[\left[z^{1 / T}, z^{-1 / T}\right]\right]$ sending $u \in V^r(0 \leq r \leq T-1)$ to $Y_M(u, z)=\sum_{n \in r / T+\mathbb{Z}} u_n z^{-n-1}$ such that the following conditions hold:
 
{\rm (1) }$Y_M(\mathbf{1}, z)=\operatorname{id}_M$.

{\rm(2)} For $u \in V^r$ and $w \in M, u_{r / T+n} w=0$ for $n \gg 0$.

{\rm(3)} For any $u \in V^r, v \in V^s$ and $l \in \mathbb{Z}, m \in r / T+\mathbb{Z}, n \in s / T+\mathbb{Z}$, the Jacobi identity hold:
$$
  \sum_{i \geq 0}(-1)^i\binom{l}{i}\left(u_{m+l-i} v_{n+i}-(-1)^l v_{n+l-i} u_{m+i}\right)=\sum_{i \geq 0}\binom{m}{i}\left(u_{l+i} v\right)_{m+n-i} . 
$$
\end{defi}

\begin{defi}
 An admissible $g$-twisted $V$-module is a $(1/T) \N$-graded weak $g$-twisted $V$-module $
M=\bigoplus_{n \in (1/T) \N} M(n)
$
satisfying
$
v_m M(n) \subseteq M(n+\operatorname{wt} v-m-1)
$
for any $v \in V, m \in (1/T) \mathbb{Z}$ and $ n \in (1/T) \mathbb{N}$.
\end{defi}

\section{$A_{g,n}(V)\!-\!A_{g,m}(V)$-Bimodules $A_{g,n,m}(V)$}

For any weak $g$-twisted \( V \)-module \( M \) and \( n \in (1/T)\mathbb{Z} \), we define a linear map \( o_{n}(\cdot): V \to \operatorname{End} M \) by
\(
o_{n}(v) = v_{\operatorname{wt}v - 1 + n},
\) and set $o(\cdot)=o_0(\cdot)$. For \( n, m \in (1/T)\mathbb{N}  \), define
\begin{align*}
   &\Omega_{n}(M) = \left\{ w \in M \mid o_{n+i}(v)w = 0 \text{ for all } v \in V \text{ and } 0 < i \in (1/T)\mathbb{Z} \right\},\\
   &\mathcal{O}_{g,n, m}(V) = \left\{ u \in V \mid o_{m-n}(u)|_{\Omega_m(M)} = 0 \text{ for all weak }g\text{-twisted } V \text{-modules } M \right\}.
\end{align*}
And set \(\mathcal{O}_{g,n}(V)=\mathcal{O}_{g,n,n}(V)\).

For any $n \in(1 / T) \N,$ set $\bar{n}=(n-\lfloor n\rfloor)T$, where $\lfloor\cdot\rfloor$ is
the floor function.  For $0\leq r\leq T-1$,  define $\delta_{i}(r)=1$ if $r \leq i \leq T-1$ and $\delta_{i}(r)=0$ if $i<r$; and set $\delta_{i}(T)=0$ for $0\le i\le T-1$. 

For  $u \in V^r, v\in V$ and $m,n,p \in (1/T)\mathbb{N}$, set $d=\lfloor m\rfloor+\lfloor n\rfloor-\lfloor p\rfloor-1+\delta_{\bar{m}}(r)+\delta_{\bar{n}}(T-r)$,  define the product $ *_{g, m, p}^{n}$ on $V$ as follows:
$$
u *_{g, m, p}^{n} v=\sum_{i=0}^{\lfloor p\rfloor}(-1)^{i}
\binom{d+i}{i}
 \operatorname{Res}_{z} \frac{(1+z)^{\operatorname{wt} u-1+\lfloor m\rfloor+\delta_{\bar{m}}(r)+r / T}}{z^{d+1+i}}
Y(u, z) v,$$
if   $\bar{p}-\bar{n}\equiv r\ \operatorname{mod}\ T$; and
$u *_{g, m, p}^{n} v=0$ otherwise. Set  $\bar{*}_{g, m}^{n}= *_{g, m, n}^{n}$,   $ *_{g, m}^{n}=*_{g, m, m}^{n}$ and $*_{g,n}=*_{g, n, n}^{n}$. 
Define
$$O_{g, n, m}^{\prime}(V)={\rm span}\{u \circ_{g, m}^{n} v\mid u,v\in V\}+L_{n,m}(V),$$ where $L_{n,m}(V)={\rm span}\{(L(-1)+L(0)+m-n)u\mid u\in V\}$ and  for  $u \in V^{r}$, $v \in V,$
$$
u \circ_{g, m}^{n} v=\operatorname{Res}_{z} \frac{(1+z)^{\operatorname{wt} u-1+\delta_{\bar{m}}(r)+\lfloor m\rfloor+r / T}}{z^{\lfloor m\rfloor+\lfloor n\rfloor+\delta_{\bar{m}}(r)+\delta_{\bar{n}}(T-r)+1}} Y(u, z) v.
$$
Set $O_{g, n}(V)=O_{g, n, n}^{\prime}(V)$, $A_{g,n}(V)=V/O_{g,n}(V)$.
For any $a, b, c, u \in V$ and any $p_{1}, p_{2}, p_{3} \in (1/T)\mathbb{N}  $, let $O_{g, n, m}^{\prime \prime}(V)$ be the linear span of
\begin{equation*}
u *_{g, m, p_{3}}^{n}\big((a *_{g, p_{1}, p_{2}}^{p_{3}} b) *_{g, m, p_{1}}^{p_{3}} c-a *_{g, m, p_{2}}^{p_{3}}(b *_{g, m, p_{1}}^{p_{2}} c)\big).
\end{equation*}
Set
$
O_{g, n, m}^{\prime \prime \prime}(V)=\sum_{p_{1}, p_{2} \in (1/T) \N}\left(V *_{g, p_{1}, p_{2}}^{n} O_{g, p_{2}, p_{1}}^{\prime}(V)\right) *_{g, m, p_{1}}^{n} V
$, 
$
O_{g, n, m}(V)=O_{g, n, m}^{\prime}(V)+O_{g, n, m}^{\prime \prime}(V)+O_{g, n, m}^{\prime \prime \prime}(V)$ and $A_{g,n,m}(V)=V/O_{g,n,m}(V).
$ 

From \cite{DLM3} and \cite{DJ2}, we have:
\begin{thm} \label{thm3.1}
    \text{\rm (1)} The product \( *_{g,n} \) induces an associative algebra structure on \( A_{g,n}(V) \) with the identity element given by \( \mathbf{1} + O_{g,n}(V) \).
    
    \text{\rm (2)} For a weak $g$-twisted \( V \)-module \( M \), \( \Omega_{n}(M) \) is an \( A_{g,n}(V) \)-module induced by the map \( a \mapsto o(a) \) for \( a \in V^{{0}} \). If \( M = \bigoplus_{k \in (1/T) \mathbb{N}} M(k) \) is an admissible $g$-twisted  \( V \)-module, then \( \bigoplus_{0 \leq k \in (1/T) \mathbb{Z} \leq n} M(k) \subseteq \Omega_{n}(M) \), and \( M(k) \) is an \( A_{g,n}(V) \)-module  for \( 0 \leq k \in (1/T) \mathbb{Z} \leq n \).

    {\rm (3)}  For any $A_{g,n}(V)$-module $U$, there exists
	an admissible  $g$-twisted  $V$-module $\bar{M}(U)$ such that $\bar{M}(U)(n) = U$.

    {\rm(4)} \( A_{g,n, m}(V) \) is an \( A_{g,n}(V)\!-\!A_{g,m}(V) \)-bimodule for \( m, n \in (1/T) \mathbb{N}  \), where the left and right actions of \( A_{g,n}(V) \) and \( A_{g,m}(V) \) are induced by \( \bar{*}^n_{g,m} \) and \( *^n_{g,m} \), respectively.
\end{thm}

Set
$
\mathcal{M}=\bigoplus_{n \in (1/T)\mathbb{N} } A_{g, n, m}(V),
$
then $\mathcal{M}$ is $(1/T) \mathbb{N} $-graded such that $\mathcal{M}(n)=A_{g, n, m}(V)  $. For $u \in$ $V^{r}, p\in r/T+\mathbb{Z}$  and   $n \in (1/T) \mathbb{N},$ define an operator $u_{p}$ from $\mathcal{M}(n)$ to $\mathcal{M}(n+\operatorname{wt} u-p-1)$ by
\[
u_{p}\big(v + O_{g,n,m}(V)\big) =
\begin{cases}
u *_{g, m, n}^{\operatorname{wt}u - p - 1 + n} v + O_{g,\operatorname{wt}u - p - 1 + n,\,m}(V),
& \text{if } \operatorname{wt} u - 1 - p + n \geq 0, \\
0,
& \text{if } \operatorname{wt} u - 1 - p + n < 0,
\end{cases}
\]
for $v \in V$. Then we form a generating function
    $Y_{\mathcal{M}}(u,z)=\sum_{p\in(1/T)\Z} u_pz^{-p-1}$. And \( \mathcal{M} \) is an admissible $g$-twisted \( V \)-module by \cite[Theorem 5.12]{DJ2}.

\begin{thm}\label{thm3.2}
For any \( n, m \in (1/T) \mathbb{N}  \), $O_{g,n}(V) = \mathcal{O}_{g,n}(V)$ and $ O_{g,n, m}(V) = \mathcal{O}_{g,n, m}(V).$
\end{thm}

\begin{proof}
By Theorem~\ref{thm3.1}(2), we have
$ A_{g,m, m}(V) = \mathcal{M}(m) \subseteq \Omega_{m}(\mathcal{M}) .$ For any \( u \in \mathcal{O}_{g,n, m}(V) \), by the definition of \( \mathcal{O}_{g,n, m}(V) \) and Theorem \ref{thm3.1}(4), we have
\[
0 = o_{m-n}(u) \left( \mathbf{1} + O_{g,m, m}(V) \right) = u *_{g,m}^n \mathbf{1} + O_{g,n, m}(V) = u + O_{g,n, m}(V),
\]
which implies \( \mathcal{O}_{g,n, m}(V) \subseteq O_{g,n, m}(V) \). By  \cite[Lemma 5.2]{DJ2}, $
O_{g,n, m}(V) \subseteq \mathcal{O}_{g,n, m}(V).
$ Thus \( \mathcal{O}_{g,n, m}(V) = O_{g,n, m}(V) \).
Consider the admissible $g$-twisted $V$-module $\bar{M}(A_{g,n}(V))$ from Theorem \ref{thm3.1}(3), so $\bar{M}(A_{g,n}(V))(n) = A_{g,n}(V)\subseteq \Omega_n(\bar{M}(A_{g,n}(V)))$ by Theorem \ref{thm3.1}(2). For any $u \in \mathcal{O}_{g,n}(V)$ , we have
$$
0 = o(u)(\mathbf{1} + O_{g,n}(V)) = u *_{g,n} \mathbf{1} + O_{g,n}(V) = u + O_{g,n}(V),
$$
which implies $\mathcal{O}_{g,n}(V) \subseteq O_{g,n}(V)$, then \( O_{g,n}(V) = \mathcal{O}_{g,n}(V) \).
\end{proof}

\section{Universal enveloping algebra $U(V[g])$ of $V$ with respect to $g$}
In this section, we shall recall the construction of the universal enveloping algebra  $U(V[g])$ of $V$ with respect to $g$ from~\cite{HX2}.
Recall from~\cite{DLM2}  the Lie algebra
$$\hat{V}[g]=\mathcal{L}(V, g) / D \mathcal{L}(V, g),$$
where $\mathcal{L}(V, g)=\bigoplus_{r=0}^{T-1} V^{r} \otimes \mathbb{C} t^{\frac{r}{T}}\left[t, t^{-1}\right]$ and $D=L(-1) \otimes \operatorname{id}_V+\operatorname{id}_V \otimes \frac{d}{dt} .$ Denote by $u(m)$ the
image of $u \otimes t^{m}$ in $\hat{V}[g] .$ Then the Lie bracket on $\hat{V}[g]$ is given by
$$
\left[u(m+{r}/{T}), v(n+{s}/{T})\right]=\sum_{i=0}^{\infty}\dbinom{m+{r}/{T}}{i}(u_{i} v)(m+n+{(r+s)}/{T}-i)
$$
for $u \in V^{r}, v \in V^{s}$ and $m, n \in \mathbb{Z}$. If we define the degree of $u(m)$ to be $\operatorname{wt} u-m-1,$ then $\hat{V}[g]$ is a $(1 / T) \mathbb{Z}$-graded Lie algebra, i.e., $\hat{V}[g]=\bigoplus_{m \in(1 / T) \mathbb{Z}} \hat{V}[g]_{m}$ and $\left[\hat{V}[g]_{i}, \hat{V}[g]_{j}\right]\subseteq\hat{V}[g]_{i+j}$
for any $i, j \in(1 / T) \mathbb{Z} .$
Let $U(\hat{V}[g])$  be the universal enveloping algebra of the Lie algebra $\hat{V}[g]$. Then the $(1 / T) \mathbb{Z}$-grading on $\hat{V}[g]$ induces a $(1 / T) \mathbb{Z}$-grading on $U(\hat{V}[g])=\bigoplus_{m \in(1 / T) \mathbb{Z}} U(\hat{V}[g])_{m} .$ Set
$$
U(\hat{V}[g])_{m}^{k}=\sum_{ k\geq  i \in(1 / T) \mathbb{Z}} U(\hat{V}[g])_{m-i} U(\hat{V}[g])_{i}
$$
for $ 0> k \in(1 / T) \mathbb{Z}$ and $U(\hat{V}[g])_{m}^{0}=U(\hat{V}[g])_{m}.$ Then
$
U(\hat{V}[g])_{m}^{k} \subseteq U(\hat{V}[g])_{m}^{k+1/T}
$
and
$$
\bigcap_{k \in-(1 / T) \mathbb{N}} U(\hat{V}[g])_{m}^{k}=0, \quad \bigcup_{k \in-(1 / T) \mathbb{N}} U(\hat{V}[g])_{m}^{k}=U(\hat{V}[g])_{m}.
$$
Thus $\left\{U(\hat{V}[g])_{m}^{k} \mid k \in-(1 / T) \mathbb{N}\right\}$ forms a fundamental neighborhood system of $U(\hat{V}[g])_{m}$. Let $\widetilde{U}(\hat{V}[g])_{m}$ be the completion of $U(\hat{V}[g])_{m}$, then
$\widetilde{U}(\hat{V}[g]):=\bigoplus_{m \in(1 / T) \mathbb{Z}} \widetilde{U}(\hat{V}[g])_{m}$ is a complete topological ring which allows infinite sums in it.

For each $m \in(1 / T) \mathbb{Z},$ define a linear map $J_{m}(\cdot): V \rightarrow \hat{V}[g]$ sending
$u \in V^{r}$ to $u(\operatorname{wt} u+m-1)$ if $m \in r / T+\mathbb{Z}$  and zero otherwise.
\begin{defi}
		The universal enveloping algebra $U(V[g])$ of $V$ with respect to $g$ is the quotient of $\widetilde{U}(\hat{V}[g]) $ by the two-sided ideal generated by the relations:
		$
		\mathbf{1}(i)=\delta_{i,-1}  \text { for } i \in \mathbb{Z}$
		and
	\begin{align*}
	&\sum_{i \geq 0}(-1)^{i}\dbinom{l}{i}\left(J_{s-i}(u) J_{t+i}(v)-(-1)^{l} J_{l+t-i}(v) J_{s+i-l}(u)\right) \\
	=&\sum_{i \geq 0} \dbinom{s+\operatorname{wt}u-l-1}{i}
	J_{s+t}\left(u_{l+i} v\right)\quad\text{ for } u\in V^r, v \in V^{r^{\prime}}, l \in \mathbb{Z}, s \in \frac{r}{T}+\mathbb{Z}, t\in \frac{r^{\prime}}{T}+\mathbb{Z}. \notag
 	\end{align*}
\end{defi}

It is clear that $U(V[g])=\bigoplus_{m \in(1 / T) \mathbb{Z}} U(V[g])_{m}$ is  a $(1/T)\mathbb{Z}$-graded associative algebra. Set
$$
U({V}[g])_{m}^{k}=\sum_{k \geq  i \in(1 / T) \mathbb{Z}} U({V}[g])_{m-i} U({V}[g])_{i}
$$
for $0>k \in(1/T)\Z$. Then  $U(V[g])_{n-m} / U(V[g])_{n-m}^{-m-1/T}$ is a $U(V[g])_{0} / U(V[g])_{0}^{-n-1/T}\!-\!U(V[g])_{0} / U(V[g])_{0}^{-m-1/T}$-bimodule for any $n,m\in(1/T)\N$.

\begin{rmk} \label{Remark 4.2}
	\rm{(1)} From the construction of $U(V[g])$, any weak $g$-twisted $V$-module is naturally a $U(V[g])$-module with the action induced by the map $u(m) \mapsto u_{m}$ for any $u \in V^r$ and $m \in r/T+\mathbb{Z}$.

    {\rm(2)}  We shall continue to denote by $J_s(u)$ the image of the element $J_s(u) \in \widetilde{U}(\hat{V}[g])$ in $U(V[g])$ or its quotients.
\end{rmk}

\section{The proof of Theorem \ref{thm1.1}}

Before stating the main result, we first need to present two lemmas.
\begin{lem}\label{lem8.3}
Let $n, m \in (1/T)\N$. For any element
\[
w = \sum J_{m_1}(u_1) \cdots J_{m_k}(u_k) \in U(V[g])_{n - m} \big/ U(V[g])_{n - m}^{-m - 1/T},
\]
where $u_j \in V$ and $m_j \in (1/T)\Z$, there exists $u_{k+1} \in V$ such that
$
w = J_{m - n}(u_{k+1}).
$
\end{lem}

\begin{proof}
We prove by induction on the lexicographical order of pairs $(k, -m_k)$ that for any monomial $J_{m_1}(u_1) \cdots J_{m_k}(u_k)$, there exists $u_{k+1} \in V$ such that
\[
J_{m_1}(u_1) \cdots J_{m_k}(u_k) \equiv J_{m - n}(u_{k+1}) \bmod U(V[g])_{n - m}^{-m - 1/T}.
\]
The claim is clear when $k = 1$ or when $m_k > m$. Suppose, for contradiction, that the statement fails for some monomial, and let $(l, -m_l)$ be a minimal counterexample with respect to lexicographical order. Write
$
s = m_{l-1}$,  $t = m_l$,  $u = u_{l-1}$,  $v = u_l,
$
and denote by $J(l-2)$ the product $J_{m_1}(u_1) \cdots J_{m_{l-2}}(u_{l-2})$. Using the defining relation of $U(V[g])$, we compute
\begin{align*}
&\,J(l-2) J_s(u) J_t(v)\\
=& -\sum_{k \geq 1} (-1)^k \binom{s - \xqz{m} - 1}{k} J(l-2) J_{s - k}(u) J_{t + k} (v) \\
& + \sum_{k \geq 0} (-1)^{k + s - \xqz{m} - 1} \binom{s - \xqz{m} - 1}{k} J(l-2) J_{s + t - \xqz{m} - k - 1}(v) J_{\xqz{m} + k + 1} (u) \\
& + \sum_{k \geq 0} \binom{\wt u  +\xqz{m} }{k}   J(l-2) J_{s + t}(u_{k+s-\xqz{m}-1} v).
\end{align*}
Modulo $U(V[g])_{n - m}^{-m - 1/T}$, the second term vanishes because $\xqz{m} + k + 1 > m$ for all $k \geq 0$. Thus,
\begin{align*}
&\,J(l-2) J_s(u) J_t(v)\\
\equiv & -\sum_{k \geq 1} (-1)^k \binom{s - \xqz{m} - 1}{k} J(l-2) J_{s - k}(u) J_{t + k} (v) \\
& + \sum_{k \geq 0} \binom{\wt u  +\xqz{m} }{k}   J(l-2) J_{s + t}(u_{k+s-\xqz{m}-1} v)\bmod U(V[g])_{n - m}^{-m - 1/T}.
\end{align*}
Each term on the right-hand side corresponds to a pair strictly smaller than $(l, -m_l)$ in lexicographical order: the first sum involves $(l, -m_l - k)$ with $k \geq 1$, and the second involves $(l - 1, -m_l - m_{l-1})$. This contradicts the minimality of $(l, -m_l)$, completing the proof.
\end{proof}

By \cite[Lemma 5.1]{DJ2}, we obtain the following result.

\begin{lem}\label{lem8.4}
Let $u, v \in V$ and $m, n, p \in (1/T)\N$. Then
\[
J_{m-n}\bigl(u *_{g,m,\,p}^{\,n} v\bigr) 
\equiv J_{p - n}(u)J_{m - p} (v)
\mod U(V[g])_{n - m}^{-m - 1/T}.
\]
\end{lem}

 Set
$\mathbb{M}=\bigoplus_{n \in (1/T)\mathbb{N} } U(V[g])_{n-m} / U(V[g])_{n-m}^{-m-1/T}$, then
 $\mathbb{M} $ is $(1/T)\mathbb{N}$-graded such that   $\mathbb{M}(n)=U(V[g])_{n-m} / U(V[g])_{n-m}^{-m-1/T}$.
We equip $\mathbb{M} $ with the vertex operator maps $Y_{\mathbb{M} }(u, z)=\sum_{p \in r/T+\mathbb{Z}} u_{p} z^{-p-1}$ for  $u\in V^r$, where for $n \in (1/T)\mathbb{N} $, the linear map $u_p$ from $\mathbb{M} (n)$ to $\mathbb{M} (n+\operatorname{wt} u-p-1)$ is defined as follows:
$$u_{p}(v  )=\left\{\begin{array}{ll}
u(p) v  , & \text { if } n+\operatorname{wt}u-p-1 \geq 0, \\
0, & \text { if } n+\operatorname{wt}u-p-1<0,
\end{array}\right.$$
for $v \in U(V[g])_{n-m} / U(V[g])_{n-m}^{-m-1/T}$.  Then $\mathbb{M} $ is an admissible $g$-twisted $V$-module, since the twisted Jacobi identity follows immediately from the construction of $U(V[g])$.

Now we begin to present the proof of Theorem \ref{thm1.1}.

\begin{proof}
Since $\mathbb{M}(m) = U(V[g])_{0} / U(V[g])_{0}^{ -m- 1/T} \subseteq {\Omega}_m(\mathbb{M})$, for any $u \in {O}_{g,n,m}(V)$, it follows from Theorem~\ref{thm3.2} that
$
0 = o_{m-n}(u)(\mathbf{1}(0) + U(V[g])_0^{ -m - 1/T}) = J_{m-n}(u) + U(V[g])_{n-m}^{-m - 1/T}.
$
Hence, $\varphi_{n,m}$ is well-defined.
Surjectivity of $\varphi_{n,m}$ follows from Lemma~\ref{lem8.3}. Suppose $u \in V$ satisfies
$
J_{m - n}(u) \in U(V[g])_{n - m}^{-m - 1/T}.
$
Then, by Remark~\ref{Remark 4.2}(1), the operator $o_{m-n}(u)$ acts as zero on ${\Omega}_m(M)$ for every weak $g$-twisted   $V$-module $M$. By Theorem~\ref{thm3.2}, this implies $u \in {O}_{g,n,m}(V)$. Hence  $\varphi_{n,m}$ is injective.
Let $u, v \in V$, using Lemma~\ref{lem8.4}, we compute
\begin{align*}
&\varphi_{n,m}\!\left( \bigl(u + {O}_{g,n,m}(V)\bigr) *_{g,m}^{\,n} \bigl(v + {O}_{g,m}(V)\bigr) \right) \\
 =& \varphi_{n,m}\!\left( u *_{g,m}^{\,n} v + {O}_{g,n,m}(V) \right)  
 =  J_{m - n}\bigl(u *_{g,m}^{\,n} v\bigr)  + U(V[g])_{n - m}^{-m - 1/T} \\
 =& J_{m - n}(u) J_{0}(v) + U(V[g])_{n - m}^{-m - 1/T} 
 =  \bigl(J_{m - n}(u) + U(V[g])_{n - m}^{-m - 1/T}\bigr) \cdot \bigl(J_{0}(v) + U(V[g])_0^{-m - 1/T}\bigr) \\
 = &\varphi_{n,m}\!\left(u + {O}_{g,n,m}(V)\right) \cdot \varphi_{m,m}\!\left(v + {O}_{g,m}(V)\right).
\end{align*}
When $m = n$, we have ${A}_{g,n}(V) = {A}_{g,n,n}(V)$ by Theorem~\ref{thm3.2}, and the above computation shows that $\varphi_{n,n}$ preserves multiplication. Since it is also bijective, $\varphi_{n,n}$ is an algebra isomorphism.
And $\varphi_{n,m}$ is a right ${A}_{g,m}(V)$-module homomorphism. A symmetric argument shows it is also a left ${A}_{g,n}(V)$-module homomorphism.
This completes the proof.
\end{proof}

\section*{Acknowledgment}
This work is supported by the National Natural Science Foundation of China (No. 12271406). The author is grateful to his supervisor Professor Jianzhi Han for his guidance.


\end{document}